\documentclass{article}
\usepackage{latexsym, amsthm, amssymb, amsmath, amstext, amsfonts, verbatim, enumerate, graphicx}
\usepackage[latin1]{inputenc}

\newcommand{\set}[1]{\ensuremath{\left\{#1\right\}}}
\newcommand{\abs}[1]{\ensuremath{|#1|}}

\newcommand{\rand}{\partial}

\newcommand{\diam}{\textnormal{diam}}

\theoremstyle{definition}
\newtheorem{Def}{Definition}[section]
\newtheorem{Exam}[Def]{Example}
\newtheorem{Rem}[Def]{Remark}

\theoremstyle{plain}
\newtheorem{Lem}[Def]{Lemma}
\newtheorem{Cor}[Def]{Corollary}
\newtheorem{Prop}[Def]{Proposition}
\newtheorem{Tm}[Def]{Theorem}

\newtheorem{Claim}[Def]{Claim}

\newcommand{\nat}{{\mathbb N}}
\newcommand{\real}{{\mathbb R}}

\newcommand{\ganz}{{\mathbb Z}}

\newcommand{\BF}{\ensuremath{\mathcal B}}

\newcommand{\UF}{\ensuremath{\mathcal U}}
\newcommand{\VF}{\ensuremath{\mathcal V}}

\newcommand{\YF}{\ensuremath{\mathcal Y}}
\newcommand{\ZF}{\ensuremath{\mathcal Z}}

\begin{document}

\title{Spanning trees in hyperbolic graphs}
\author{Matthias Hamann\\\\Fachbereich Mathematik\\Universit\"at Hamburg}
\date{}
\maketitle

\begin{abstract}
We construct spanning trees in locally finite hyperbolic graphs that represent their hyperbolic compactification in a good way: so that the tree has at least one but a bounded number of disjoint rays to each boundary point. The bound depends only on the (Assouad) dimension of the boundary. As a corollary we extend a result of Gromov which says that from every hyperbolic graph with bounded degrees one can construct a tree (disjoint from the graph) with a continuous surjection from the ends of the tree onto the hyperbolic boundary such that the surjection is finite-to-one. We shall construct a tree with these properties as a subgraph of the hyperbolic graph, which in addition is also a spanning tree of that graph.
\end{abstract}

\section{Introduction}

A spanning tree of a graph is called {\em end-faithful} if the tree contains exactly one ray from each end, starting at the root.
Halin \cite{halin64} proved that every countable graph has an end-faithful spanning tree.
Examples for such trees are the normal spanning trees (see \cite{BrochetDiestel,Wurzelbaeume} and \cite[Chapter~8]{GelbesBuch}).
So it is a natural question to ask~-- if we replace the end-compactification of a graph by other compactifications that refine the end-compactification~-- how we can expect a spanning tree to behave with respect to the new compactification:
Is it possible that the ends of a spanning tree represent the boundary points of this compactification in a one-to-one correspondence?

In this paper we study such a generalization of end-faithful spanning trees to spanning trees in locally finite hyperbolic graphs, replacing the end-compacti\-fi\-cation by the hyperbolic compactification.

A {\em hyperbolic graph} $G$ is a connected graph for which there exists a $\delta$ such that for every three vertices every {\em geodesic} between two of them, that is, every path representing the distance between them, is contained in a $\delta$-neighbourhood of the union of every two geodesics between each other two of those vertices.
A {\em hyperbolic boundary point} is an equivalence class of the following equivalence relation (compare with \cite[(22.12)]{woessBook}) of geodesic rays:
two geodesic rays $x_0x_1\ldots$ and $y_0y_1\ldots$ are {\em equivalent} if $\liminf_{i\to\infty}d(x_i,y_i)$ is finite.
The {\em hyperbolic boundary} $\rand G$ is the set of all hyperbolic boundary points.
This is one of many equivalent definitions of the hyperbolic boundary (see \cite{GhHaSur,HLV-DirichletProblem,woessBook} and Section~\ref{hypGraphs} of this paper).
Let $\widehat{G}:=G\cup\rand G$.

In~\cite[Section~7]{BDS-Embedding} (see also \cite{BourdonPajot,Elek-Cohomology} or \cite[Chapter~6]{BS-Elements}) for every compact metric space $X$, a construction of a locally finite hyperbolic graph $G$ is given such that $\rand G$ is homeomorphic to~$X$.
But as the end space of any tree is a totally disconnected topological space (see \cite{jung71}) there cannot be any spanning tree such that the induced map from the boundary of the tree to the boundary of the graph is a homeomorphism.
Hence for a given locally finite hyperbolic graph $G$ there is not always a tree $T$ such that the end space of~$T$ is homeomorphic to the hyperbolic boundary of~$G$.
We shall give an explicit example for such a situation in Section~\ref{SpannTreesSection} (Example~\ref{ex_introEx}).
But whenever the identity of a hyperbolic graph $G$ extends to a homeomorphism from $\widehat{G}$ to $G$ with its ends, any normal spanning tree~-- or more generally any end-faithful spanning tree~-- is faithful with respect to the hyperbolic boundary points.

Hyperbolic graphs in which the notion of hyperbolic boundary points and ends coincide are for example all locally finite graphs quasi-isometric to a tree (see \cite{KM-QuasiIsometries}) or~-- more generally, compare with \cite[Theorem~2.8]{KM-QuasiIsometries}~-- graphs in which any end is a thin end in the sense of~\cite[Chapter~8]{GelbesBuch}, as any end of a locally finite hyperbolic graph that consists of more than one hyperbolic boundary point consists of uncountably many hyperbolic boundary points since this set of hyperbolic boundary points is a connected set \cite[Proposition~7.5.17]{GhHaSur}.
Thus, in locally finite hyperbolic graphs, the hyperbolic boundary is a refinement of the set of its ends and it is furthermore a compact metric space \cite[Proposition~7.2.9]{GhHaSur}.
This is not the case for arbitrary graphs: neither the hyperbolic boundary has to be compact for hyperbolic graphs that are not locally finite nor it is a refinement of the set of ends of such graphs.
Because of this, we restrict our point of view to locally finite graphs.

Instead of spanning trees that are faithful to boundary points, we may perhaps hope that we get spanning trees that have only a finite number of distinct paths from the root to each boundary point such that the set of these numbers is bounded.
For hyperbolic Cayley graphs, this is known to be true (cp.\ \cite[p.\ 10]{KB-BoundaryHypGroup}):

\begin{Tm}\label{GromovAsCor}
Let $\Gamma$ be a locally finite hyperbolic Cayley graph. Then there exists a rooted geodesic spanning tree $T$ of~$\Gamma$ such that its embedding extends uniquely to a continuous bounded-to-one map from $\widehat{T}$ to $\widehat{\Gamma}$.\qed
\end{Tm}

\emph{Geodesic spanning trees} are spanning trees that preserve for each vertex the distance to the root from the distance-metric of the graph.
The general idea of the proof of Theorem~\ref{GromovAsCor} is the following.
Take an order on a finite set $S$ of generators, and take the tree consisting of edges that lie on shortest words with respect to the following order: we set $w < u$ if either $|w| < |u|$ or if $|w| = |u|$ and $w$ is lexicographically smaller than~$u$.
The resulting tree has the claimed properties.

Unfortunately, we cannot take such a geodesic spanning tree in general to obtain a continuous bounded-to-one map from $\widehat{T}$ to~$\widehat{G}$: in Example~\ref{SecondExample}, we shall discuss a locally finite hyperbolic graph with precisely one hyperbolic boundary point such that each of its rooted geodesic spanning trees has infinitely many ends.
However, we shall obtain a lower bound on the maximum number of tree ends mapping to a common hyperbolic boundary point in Section~\ref{GeodSpannTreesSection}.
This bound depends only on the topological dimension of the hyperbolic boundary:

\begin{Tm}\label{TheoremGeodTree}
Let $G$ be a locally finite hyperbolic graph whose boundary has topological dimension $n\in\nat$. Then for every rooted geodesic spanning tree $T$ of~$G$ there is a boundary point $\eta\in\rand G$ with at least $n+1$ distinct rays starting at the root and converging to~$\eta$.
\end{Tm}

If we allow arbitrary trees, then it is possible to obtain a positive result as soon as the hyperbolic graph has bounded degree.
Gromov \cite[\S 7.6]{gromov} states the following theorem:

\begin{Tm}\label{GromovTm}
Let $X$ be a $\delta$-hyperbolic graph with maximum degree ${N<\infty}$. Then there is a locally finite tree $T(X)$ with maximum degree at most\linebreak $\exp(\exp((\delta+1)N))$ with a continuous surjection $\rand T\to\rand X$ that is finite-to-one; additionally a boundary point of $X$ has at most $\exp(\exp(\exp((\delta+1)N)))$ preimages.\qed
\end{Tm}

Gromov constructed the tree $T(X)$ independently of the local structure of the graph $X$, just depending on the metric of $X$.
Thus a vertex in $T(X)$ may have higher degree than all vertices in $X$.

Coornaert and Papadopoulos \cite[Chapter~5]{SymbolicDynamics} worked out several ideas of Gromov \cite[Sections~7.6, 8.5.B, and 8.5.C]{gromov}~-- one of which is Theorem~\ref{GromovTm}~-- and thereby constructed similar abstract trees.
They constructed three different abstract trees, two of which also have a continuous finite-to-one surjection from the ends of the tree to the hyperbolic boundary of the graph.

We shall prove an extension of Gromov's theorem in Section~\ref{SpannTreesSection}.
The tree we shall construct will be a spanning tree of the hyperbolic graph:

\begin{Tm}\label{SpanningTrees}
Let $G$ be a locally finite hyperbolic graph whose boundary $\rand G$ has finite Assouad dimension. Then there exists an $n\in\nat$, depending only on the dimension, and a rooted spanning tree $T$ of $G$, with the following properties:
\begin{enumerate}[(i)]
\item Every ray in~$T$ converges to some point in the boundary of~$G$;
\item for every boundary point $\eta$ of~$G$ there is a ray in~$T$ converging to~$\eta$;
\item for every boundary point $\eta$ of~$G$ there are at most $n$ distinct rays in~$T$ that start at the root of~$T$ and converge to~$\eta$.
\end{enumerate}
\end{Tm}

Examples of locally finite hyperbolic graphs whose hyperbolic boundary has finite Assouad dimension are all graphs with bounded degree (see \cite[Theorem~9.2]{BonkSchramm_Embeddings}), so in particular all Cayley graphs of finitely generated groups with respect to a finite set of generators.
Thus, we obtain the important part of Theorem~\ref{GromovTm} as a corollary of Theorem~\ref{SpanningTrees}.

In contrast to end-faithful spanning trees where the normal trees of graphs form a class of graphs that are all end-faithful, there is no generic class of spanning trees known that always fulfill the conclusions of Theorem~\ref{SpanningTrees}.

A direct consequence of~\cite[Theorem 1.5]{BenjaminiSchramm-Cheeger} is the following theorem:

\begin{Tm}
Let $\Gamma$ be a non-elementary Gromov hyperbolic group, $S$ a finite generating set of $\Gamma$, and $G$ the corresponding Cayley graph.
Then $G$ contains a bilipschitz image of the binary tree $T_2$, in particular such that the induced mapping $\rand T_2\to\rand G$ is a homeomorphic embedding of $\rand T_2$ into $\rand G$.\qed
\end{Tm}

So in any connected locally finite hyperbolic Cayley graph (which has exponential growth, cp.~\cite[6.14]{B-ACourse}), there is a bilipschitz embedding of a tree of exponential growth.
Hence, this theorem has its main interest in the graph itself whereas Theorem~\ref{SpanningTrees} has its main interest in the hyperbolic boundary.
Because of this, it is a natural question to ask whether for every connected locally finite hyperbolic graph $G$, there is a tree $T$ that represents both $G$ and $\rand G$ in a suitable way.

In Section~\ref{Topology} we shall give explicit definitions of the two dimensional concepts we use, the Assouad dimension and the topological dimension, and state some of their properties. 
For a more detailled introduction to the Assouad dimension we refer to Luukkainen \cite[Appendix~A]{Luukainen}.

\section{Hyperbolic graphs}\label{hypGraphs}

Let us give a brief introduction in hyperbolic graphs. For more details we refer to \cite{ABCFLMSS,CoornDelPapa,GhHaSur,gromov} and \cite[Chapter~22]{woessBook}.

Let $G=(V,E)$ be a hyperbolic graph and let $\delta$ be a constant that realizes the hyperbolicity condition from the introduction.
We are investigating $G$ from a topological point of view so that every edge of~$G$ is understood as a homeomorphic image of the real interval $[0,1]$.

Let $o$ be a vertex in $G$.
The {\em Gromov-product} (with respect to~$o$) of two vertices $x$ and $y$ is
\[(x,y)_o:=\frac{1}{2}(d(x,o)+d(y,o)-d(x,y)).\]
If it is obvious by the context that we use $o$ as the base-point for the product, we simply write $(x,y)$.

A {\em ray} is a one-way infinite path and a {\em double ray} is a two-way infinite path.
Two rays are {\em equivalent} if no finite set of vertices separates them.
This is an equivalence relation and an {\em end} is an equivalence class of rays.
For more informations on ends of graphs we refer to~\cite{GelbesBuch,halin64,Woess-Amenable}.
A ray $x_0x_1x_2\ldots$ is {\em geodesic} if $d(x_i,x_j)=\abs{i-j}$ for all $i,j\in\nat$, and a double ray $\ldots x_{-1}x_0x_1\ldots$ is {\em geodesic} if $d(x_i,x_j)=\abs{i-j}$ for all $i,j\in\ganz$.

We are giving~-- in addition to the definition of the introduction~-- a second definition, a topological one, of the hyperbolic boundary:
a sequence $(x_i)_{i\geq 0}$ {\em converges} to infinity if $\lim_{i,j\to\infty}(x_i,x_j)=\infty$.
Two sequences $(x_i)_{i\geq 0}$, $(y_j)_{j\geq 0}$ are {\em equivalent} if $\lim_{i,j\to\infty}(x_i,y_j)=\infty$.
In hyperbolic graphs this equivalence is indeed an equivalence relation that is independent from the base point $o$ of the Gromov-product.
The hyperbolic boundary can be defined as equivalence classes of this equivalence relation.
A sequence $(x_i)_{i\geq 0}$ {\em converges} to a boundary point if it is in its equivalence class.

A third way to define the boundary is by defining a metric $d_\varepsilon$ on $G$ and then defining $\widehat{G}$ as the completion of $G$ induced by $d_\varepsilon$. Let $\varepsilon>0$ with $\varepsilon':=\exp(\varepsilon\delta)-1<\sqrt{2}-1$. Let
\[\varrho_{\varepsilon}(x,y):=\exp(-\varepsilon(x,y)),\]
\[\varrho_\varepsilon(x_0,\ldots,x_n):=\sum_{i=1}^n\varrho_\varepsilon(x_{i-1},x_i)\]
and
\[d_\varepsilon(x,y):=\inf\set{\varrho_\varepsilon(c)\mid c\text{ chain between }x\text{ and }y}.\]
It is easy to check that $d_\varepsilon$ is a metric on $G$.
In~\cite{GhHaSur} the equivalence of all three definitions of the hyperbolic boundary, which we mentioned here, is shown.

We shall now define a topology on $\widehat{G}$, which is compatible with the topology of $\widehat{G}$ which is induced by $d_\varepsilon$.
For two vertices or hyperbolic boundary points $a$ and $b$ we define the {\em Gromov-product} (once more):
\[(a,b):=\sup\liminf\limits_{i,j\to\infty}(x_i,y_j)\]
where the supremum is taken over all sequences $(x_i)_{i\geq 0}\to a$ and $(y_i)_{i\geq 0}\to b$.
Obviously, it coincides for vertices with the previous definition of the Gromov-product, so we are allowed to use the same symbol.

The following proposition is proved for the locally finite case in \cite[Proposition 4.8]{ABCFLMSS}. For arbitrary hyperbolic graphs compare also with \cite[II.8.5]{BridsonHaefliger}.

\begin{Prop}\label{prop_BasisTop}
Let $G$ be a hyperbolic graph.
The balls $B_r(x)$ for all $x\in VG$ and all $r\in\real_{\geq 0}$ and the sets $N_k(x):=\{y\in\widehat{G}|(x,y)>k\}$ for all $x\in\rand G$ and all $k\in\real_{\geq 0}$ form a basis of a topology on $\widehat{G}$.\qed
\end{Prop}

This topology is compatible with the metrics $d_\varepsilon$, which for locally finite graphs, makes the boundary a compact metric space by the following theorem.

\begin{Tm}\label{MetricTopologyOfBoundary}\label{BoundComp}{\rm \cite[Propositions 7.3.10 and 7.2.9]{GhHaSur}}
Let $G$ be a locally finite hyperbolic graph. For all $\varepsilon>0$ with $\varepsilon'=\exp(\varepsilon\delta)-1<\sqrt{2}-1$, the metric space $(\widehat{G},d_\varepsilon)$ is a compact metric space such that its metric $d_\varepsilon$ is compatible with the just defined topology in the sense that we have
\[\varepsilon'\cdot\exp(-\varepsilon\cdot(\eta,\nu))\leq d_\varepsilon(\eta,\nu)\leq\exp(-\varepsilon\cdot(\eta,\nu))\]
for all $\eta,\nu\in\rand G$.\qed
\end{Tm}

Theorem~\ref{MetricTopologyOfBoundary} shows that the definitions of the topological space are equivalent. We may thus use the direct definition of the topology or the definition via the metric depending on the situation.

We close this section with some propositions that we shall need later.

\begin{Prop}\label{RightChoiceOfGeodRays} {\rm \cite[(22.11) and (22.15)]{woessBook}}
Let $G$ be a locally finite hyperbolic graph with two distinct boundary points $\eta$ and $\nu$.
Let $o$ be a vertex, $(x_i)_{i\in\nat}$ a geodesic ray converging to $\eta$, and $(y_j)_{j\in\nat}$ a geodesic ray converging to $\nu$.
Then the following two statements hold:
\begin{enumerate}[(i)]
\item There is a geodesic ray in $G$ starting at $o$ and having only finitely many vertices different from $(x_i)_{i\in\nat}$.
\item There is a geodesic double ray having only finitely many vertices different from $(x_i)_{i\in\nat}$ and $(y_j)_{j\in\nat}$.
One side of this double ray converges to $\eta$, the other to~$\nu$.\qed
\end{enumerate}
\end{Prop}

A direct consequence of~\cite[Lemma~4.6 (4)]{ABCFLMSS} and \cite[Remark~7.2.7]{GhHaSur} is the following proposition.

\begin{Prop}\label{LimesOfGromovProd}\label{DistanceToGeod}\label{geodIn4Delta}
Let $G$ be a $\delta$-hyperbolic graph, let $\eta,\nu\in\rand G$, and let $o$ be the base-point of the Gromov-product. For all geodesic double rays $\pi$ from $\eta$ to~$\nu$ (i.e.\ one side converges to~$\eta$ and the other to~$\nu$) the following inequality holds:

\hfill$\,\,(\eta,\nu)-2\delta\leq d(o,\pi)\leq (\eta,\nu)+2\delta.$\qed
\end{Prop}

\section{Dimensions of topological spaces}\label{Topology}

Let us introduce the first dimension concept, depending only on the topology of a space.
Let $X$ be a topological space.
A {\em refinement\/} $\UF$ of an open cover $\VF$ of~$X$ is an open cover of $X$ such that for every $U\in \UF$ there is a $V\in\VF$ with $U\subseteq V$.
The space $X$ has {\em topological dimension at most $n$} if every open cover has a refinement such that each $x\in X$ lies in at most $n+1$ elements of the refinement, and $X$ has {\em topological dimension $n$} (notation: $\dim(X)=n$) if it has topological dimension at most $n$ but not topological dimension at most $n-1$.
If there exists no $n\in\nat$ such that $X$ has topological dimension at most $n$ then $X$ has {\em infinite topological dimension}.
We call an open cover $\UF$ of a topological space $X$ with topological dimension~$n$ {\em critical} if there exists no refinement $\VF$ of~$\UF$ such that each $x\in X$ lies in at most $n$ sets $V\in \VF$.

\smallskip

Let us now introduce the second dimension concept, depending on the metric of a space.
Let $X$ be a metric space.
For $\alpha,\beta>0$ let $S(\alpha,\beta)$ be the maximal cardinality of a subset $V$ of $X$ with $\alpha\leq d_X(x,y)\leq \beta$ for all $x\neq y\in V$.
Let $n$ be the infimum of all $s\geq 0$ such that there is a $C\ge 0$ with $S(\alpha,\beta)\leq C(\frac{\beta}{\alpha})^s$ for all $0<\alpha\leq\beta$.
Then $n$ is called the {\em Assouad dimension} of the metric space $X$ and we write $\dim_A(X)=n$.
If no such $s$ exists then we say that $X$ has \emph{infinite Assouad dimension}.

\smallskip

Furthermore, we introduce a property of metric spaces.
Let $X$ be a metric space. $X$ is {\em doubling} if there is an integer $\kappa\geq 1$ such that every open ball of radius $r$ can be covered by at most $2^\kappa$ open balls of radius at most $\frac{r}{2}$.
Let $\dim_2(X)$ be the infimum of all $\kappa$ such that $X$ is doubling with this $\kappa$.
A subset $Y$ of $X$ has {\em diameter} $\diam(Y):=\sup\set{d(x,y)\mid x,y\in Y}$, and a set $\YF$ of subsets of $X$ has {\em diameter} $\diam(\YF):=\sup\set{\diam(Y)\mid Y\in\YF}$.
For every $r\geq 0$, a family $\BF=(B_i)_{i\in I}$ of subsets of $X$ has {\em $r$-multiplicity} $n$ if every subset of $X$ with diameter at most $r$ intersects with at most $n$ and if one subset with diameter at most $r$ intersects with precisely $n$ members of the family.
A point $x\in X$ has {\em $r$-multiplicity} $n$ in $\BF$ if $\overline{B_r(x)}$ intersects with precisely $n$ members of the family $\BF$ non-trivially.

For a metric space $X$ it is equivalent that $X$ is doubling and that $\dim_A(X)$ is finite by the following theorem of Luukkainen.

\begin{Tm}\label{AssouadDoubling}{\rm \cite[Theorem A.3]{Luukainen}}
Let $X$ be a metric space.
Then $X$ is doubling if and only if it has finite Assouad dimension.\qed
\end{Tm}

The proof of the following lemma is similar to the proof of a lemma by Lang and Schlichenmaier \cite[Lemma 2.3]{LS-Nagata}.
But since our claim includes additional statements to the mentioned lemma and the constants vary, we give a proof of the whole lemma here.

\begin{Lem}\label{LS2.3}
Let $X$ be a doubling metric space, let $N=2^{\dim_2(X)}$, and let $r>0$.
Then $X$ has a covering $\BF$ of closed balls of diameter at most~$2r$ such that $\BF$ is the disjoint union of at most $N^4$ subsets $\BF_i$ of~$\BF$ each of which has $r$-multiplicity at most $1$; so $\BF$ has $r$-multiplicity at most~$N^4$.

Furthermore, it is possible to choose $\BF$ so that a given subset $Y$ of~$X$ with $d(x,y)>r$ for all $x,y\in Y$ is a subset of the set of centers of the balls in~$\BF$, so that each two centers have distance more than~$r$, and so that every center has $3r$-multiplicity at most $N^4$ in~$\BF$.

Additionally, if $Y$ is finite and $X$ bounded, then we may choose $\BF$ finite.
\end{Lem}

\begin{proof}
Let $Z$ be a maximal subset of~$X$ with $Y\subseteq Z$ such that $d(x,y)>r$ for each two distinct $x,y\in Z$ and let $\BF=\{\overline{B}_r(z)\mid z\in Z\}$.
We shall show that $\BF$ has the claimed properties.

By the maximality of~$Z$, we know that $\BF$ is a covering of~$X$.
Since $X$ is doubling, we know that for all $z\in Z$ the ball $\overline{B}_{4r}(z)\subseteq B_{8r}(z)$ can be covered by at most $N^4$ balls of diameter at most~$r$ each of which contains at most one element of~$Z$.
Thus, we have $|Z\cap \overline{B}_{4r}(z)|\leq N^4$ and hence every $z\in Z$ has $3r$-multiplicity at most $N^4$ in~$\BF$.
By \cite[Lemme 2.4]{Assouad} there is a map $f:Z\to\{1,\ldots, N^4\}$ with $f(x)\neq f(y)$ for all $x,y\in Z$ with $0<d(x,y)\leq 4r$.
Every $\BF_i:={\{\overline{B}_r(z)\mid z\in Z,\, f(z)=i\}}$ has $r$-multiplicity at most~$1$.
The additional statement is a direct consequence of the doubling property, so the claim follows.
\end{proof}

\begin{Rem}\label{compareDimension}
{\rm \cite[Facts 3.3]{Luukainen}}
We have $\dim(X)\leq \dim_A(X)$ for every metric space~$X$.
\end{Rem}

\section{Geodesic spanning trees in hyperbolic graphs}\label{GeodSpannTreesSection}

In this section, we shall prove Theorem \ref{TheoremGeodTree} and give an example of a locally finite hyperbolic graph that has only one boundary point but each of whose rooted geodesic spanning trees has infinitely many distinct rays starting at the root and converging to the only boundary point.
Before we prove Theorem~\ref{TheoremGeodTree}, we first prove some propositions.

\begin{Prop}\label{geodTreeClosed}
Let $G$ be a locally finite hyperbolic graph, $(T,r)$ a rooted geodesic spanning tree of $G$, and $S$ a finite set of vertices of~$G$.
Let $Z$ be the set of limit points of all rays in a connected component $C$ of $T-S$. Then $Z$ is closed in~$\rand G$.
\end{Prop}

\begin{proof}
Let $z$ be a boundary point of $G$ with $z\in\overline{Z}$ and let $(\eta_i)_{i\in\nat}$ be an infinite sequence of boundary points in~$Z$ converging to~$z$.
Let $\pi_i$ be a geodesic ray from the root of $T$ to $\eta_i$ with only finitely many vertices outside $C$.
Since $G$ is locally finite, so is $T$.
Hence there exists a ray $\pi$ in $T$ such that each edge of $T$ lies in infinitely many of the rays $\pi_i$.
Since $T$ is a geodesic tree, $\pi$ is geodesic and hence has exactly one limit point $\eta$.
Furthermore, $\pi$ has also only finitely many vertices outside $C$, as $S$ is finite and as $\{d(r,s)\mid s\in S\}$ is bounded.
Thus, $\eta$ is an element of~$Z$.
Using the Gromov-product we conclude that $\pi$ converges towards the limit point of $(\eta_i)_{i\in\nat}$ and hence $z=\eta\in Z$.
\end{proof}

\begin{Prop}\label{TwoStar}
Let $G$ be a locally finite hyperbolic graph and let $\UF$ be a finite open cover of $\rand G$. Every rooted geodesic spanning tree $(T,r)$ has the following property:
\begin{itemize}
\item[$(\dagger)$] There is a finite vertex set $S\subseteq VG$ such that for every component $C$ of $T-S$ there is a $U\in\UF$ such that every ray in $C$ converges to some $u\in U$.
\end{itemize}
\end{Prop}

\begin{proof}
Let us suppose that there is no finite vertex set $S$ satisfying $(\dagger)$.
Thus, for every finite vertex set $S$ and for one component $C$ of $T-S$ the set $Z$ of all limit points of all rays in~$C$ is not a subset of any $U\in\UF$.
We extend $S$ by the down-closure $\lceil s\rceil$ of one vertex $s$ from $C$ (that is the set of all vertices on the unique $r$-$s$ path in~$T$) to split $C$ into more than one component and thus to obtain a refinement of~$Z$.
For $S\cup\lceil s\rceil$, there is again a component $C'$ with $\rand C'\not\subseteq U$ for every $U\in\UF$.
So we may extend $S$ recursively infinitely often but in each step we extend $S$ by only finitely many vertices.
Since $T$ is locally finite, there is a ray~$\pi$ in~$T$ from which we take infinitely many vertices during this process.
Let $\eta$ be the limit point of~$\pi$ in~$\rand G$.
Then there exists a $U\in\UF$ with $\eta\in U$.
Since $U$ is open and according to Proposition~\ref{prop_BasisTop}, there is a $k\in\nat$ such that every boundary point $\nu$ of~$G$ with $(\eta,\nu)_r\geq k$ lies in~$U$.
Let $x$ be the vertex on~$\pi$ with $d(r,x)=k+3\delta+1$ and let $\mu$ be a boundary point in the closure of that component of $T-x$ that contains the ray~$\pi$ eventually.
We consider one geodesic between each two of $x,\mu$, and~$\eta$.
The chosen geodesic between $\mu$ and~$\eta$ must lie in a $\delta$-neighbourhood of the other two geodesics.
Hence, we conclude with Proposition~\ref{LimesOfGromovProd} that $(\mu,\eta)_r\geq k+1$.
This is a contradiction to the choice of~$\pi$.
\end{proof}

\begin{Prop}\label{EpsilonCircles}
Let $G$ be a locally finite hyperbolic graph and let $T$ be a rooted geodesic spanning tree of~$G$.
Assume that there exists an $m\in\nat$ such that for every $\eta\in\rand G$, there are at most $m$ distinct rays in~$T$ starting at the root of~$T$ and converging to~$\eta$.
Let $\UF$ be a finite open cover of $\rand G$, let $S$ be a finite set of vertices as in~$(\dagger)$ in Lemma~\ref{TwoStar}, let $C_i$ be the infinite components of $T-S$, let $Z_i$ be the set of all limit points of rays in~$C_i$, and let $\ZF$ be the set of all $Z_i$.
Then there exists an $\varepsilon> 0$ such that for every $\eta\in\rand G$, the ball $B_\varepsilon(\eta)$ intersects with at most $m$ elements of $\ZF$ non-trivially.
\end{Prop}

\begin{proof}
Since there are only $m$ distinct $r$-$\eta$-rays for every $\eta\in\rand G$, each boundary point is contained in at most $m$ different elements of~$\ZF$.
Let us assume that the assertion does not hold.
Then there is an infinite sequence of boundary points $(\eta_i)_{i\in\nat}$ and an infinite sequence of real numbers $(\varepsilon_i)_{i\in\nat}$ that converges to~$0$ such that every $\varepsilon_i$-neighbourhood of $\eta_i$ intersects with at least $m+1$ elements of $\ZF$ non-trivially.
The sequence of the boundary points has an accumulation point $\eta$ as $\rand G$ is compact.
Thus, we may assume that the sequence converges towards $\eta$.
Since $\ZF$ has only finitely many elements, there is a set $Z_{i_1}$ which intersects with an infinite subsequence of $(\eta_i)_{i\in\nat}$ non-trivially.
Because each $\varepsilon_i$-neighbourhood of $\eta_i$ intersects with $m+1$ elements of $\ZF$ non-trivially, we find analogously distinct $Z_{i_2},\ldots,Z_{i_{m+1}}\in\ZF$ such that each of them intersects with infinitely many $B_{\varepsilon_i}(\eta_i)$ non-trivially.
Hence, $\eta$ lies in the closure of all $Z_{i_1},\ldots,Z_{i_{m+1}}$.
The sets $Z_{i_j}$ are closed by Lemma \ref{geodTreeClosed}, so $\eta$~lies in all of them.
But this contradicts the fact that any $\mu\in\rand G$ lies in at most $m$ elements of~$\ZF$.
\end{proof}

Now we are able to prove Theorem \ref{TheoremGeodTree}.

\begin{proof}[{\itshape Proof of Theorem \ref{TheoremGeodTree}.}]
Let $\UF$ be a critical open cover of $\rand G$.
As $\rand G$ is compact, we may assume that $\UF$ is finite.
Additionally, we may assume that there is an $m$ such that $T$ contains at most $m$ distinct rays from the root to each $\eta\in\rand G$ since otherwise the theorem trivially holds.
By Proposition \ref{TwoStar}, there is a vertex set $S\subseteq V(G)$ such that for the components $C_1,\ldots,C_k$ of $G-S$ and the sets $Z_i$ of all the limit points of rays in~$C_i$, we have that every $Z_i$ is a subset of some $U\in\UF$.
Due to Proposition~\ref{geodTreeClosed}, the sets $Z_i$ are closed and due to Proposition~\ref{EpsilonCircles}, there is an $\varepsilon>0$ such that $B_\varepsilon(\eta)$ intersects with only $m$ elements of $\ZF:=\{Z_1,\ldots,Z_k\}$ non-trivially for every $\eta\in\rand G$.
Let us define for every $Z\in\ZF$ the set $Z'$ to be $Z$ together with all open $\varepsilon$-neighbourhoods around all elements of~$Z$.
Then $Z'$ is an open set.
Let $U$ be in $\UF$ with $Z\subseteq U$, and let $Z''$ be $Z'\cap U$.
Then $Z''$ is an open set, too.
Let $\VF$ be the set consisting of all the sets $Z''$ with $Z\in\ZF$.
By construction, $\VF$ is an open cover of $\rand G$ and also a refinement of $\UF$.
Since $\UF$ is critical, there is an $\eta\in\rand G$ that lies in at least $n+1$ elements of $\VF$.
For a set $Z''$ that contains~$\eta$, the set $Z$ intersects non-trivially with $B_\varepsilon(\eta)$.
Since there are at most $m$ such sets in~$\ZF$, we obtain $m\geq n+1$.
\end{proof}

We immediately get the following corollary of Theorem \ref{TheoremGeodTree}.

\begin{Cor}
Let $G$ be a locally finite hyperbolic graph whose boundary has infinite topological dimension.
Then for every rooted geodesic spanning tree $T$ there is no $n\in\nat$ such that for every boundary point $\eta\in\rand G$ there are at most $n$ distinct rays in $T$ starting at the root and converging to $\eta$.\qed
\end{Cor}

Let us finally give an example that, in general, the rooted geodesic spanning trees do not show the whole truth about the topological dimension of the hyperbolic boundary.

\begin{Exam}\label{SecondExample}
Let $V_k$ be a set of $2^k$ elements such that the $V_k$ are pairwise disjoint.
Let $G$ be a graph with vertex set $\bigcup_{k\in\nat} V_k$.
Any two vertices of the same $V_k$ are adjacent.
Furthermore, every $x\in V_k$ with $k\neq 0$ has precisely one neighbour in $V_{k-1}$, two neighbours in $V_{k+1}$, and no other neighbours.

This graph is obviously a hyperbolic graph with one end and one boundary point so its hyperbolic boundary has topological dimension~$0$.
Let $T$ be a geodesic spanning tree in $G$ with root $r$.
For every vertex $x$ in $G$ there is a subgraph $H$ of $G$ such that $H$ is isomorphic to $G$ and $x$ is mapped to the unique vertex $o\in V_0$.
If the graph with $r=o$ has infinitely many distinct $r$-$\eta$-paths for the only boundary point $\eta$, then this is also the case for any tree $T$ with arbitrary~$r$.
Thus, we may assume that $r=o$.
For every vertex $y$ there is a unique geodesic from $r$ to~$y$ and all the infinitely many geodesic rays starting at~$r$ must lie in~$T$.
This proves that every geodesic spanning tree of~$G$ has infinitely many ends each of which corresponds to the same boundary point of~$G$.
\end{Exam}

\section{Spanning trees in hyperbolic graphs}\label{SpannTreesSection}\label{1stExampleSection}

In this section we shall prove our main result, Theorem~\ref{SpanningTrees}, and deduce a corollary from that theorem.
To prove Theorem~\ref{SpanningTrees}, we shall prove in particular that the number of distinct rays to the same boundary point is finite and bounded for all hyperbolic boundary points, if the Assouad dimension of a hyperbolic boundary is finite.
Since the Assouad dimension depends on the metric of the boundary (cp.\ Theorem~\ref{MetricTopologyOfBoundary}) and since we may have chosen distinct metrics, it would be good if the existence of an upper bound does not depend on the particular metric we used for the completion of $G$.
Bonk and Schramm \cite[Section~6 and 9]{BonkSchramm_Embeddings} showed that this is indeed the case: If one hyperbolic metric $d_\varepsilon$ on $G$ induces a boundary with finite Assouad dimension, then all hyperbolic metrics have this property and all boundaries are doubling metric spaces.
But although the existence is given for each metric $d_{\varepsilon}$, the Assouad dimension may change with different metrics (compare with \cite[A.5~(9)]{Luukainen}).
Thus it may happen that for different $\varepsilon$ and $\varepsilon'$ the Assouad dimensions $\dim_A(\rand G,d_\varepsilon)$ and $\dim_A(\rand G,d_{\varepsilon'})$ may differ.
It remains open if there is another dimension concept, perhaps the topological dimension~-- recall that $\dim(X)\leq\dim_A(X)$ for all metric spaces by Remark~\ref{compareDimension}~--, that is invariant under changing the metric $d_\varepsilon$ such that the upper bound of distinct rays to the same boundary point depends only on that dimension.

\begin{proof}[Proof of Theorem~\ref{SpanningTrees}]
First, let us give a brief outline of this proof.
After adjusting some variables, we shall recursively construct an infinite rooted subtree $T'$ of~$G$ whose construction usually finishes just in the limit step.
This subtree consists only of rays starting at the root.
There will be some vertices in the graph that do not lie in~$T'$.
We shall add them to $T'$ with appropriate paths so that no new ray is created after connecting all the remaining vertices.
The obtained tree $T$ will fulfill the conclusions of Theorem~\ref{SpanningTrees}.
But before we prove this, we first show two claims about some properties of the construction of the tree $T'$.
Then we are able to show that the properties (i) to (iii) hold for the constructed subtree $T'$ and thus they also hold for the spanning tree $T$ of~$G$, as all components of $T-T'$ are finite.

\bigskip

Let $\varepsilon>0$ such that $\varepsilon':=\exp(\varepsilon\delta)-1<\sqrt{2}-1$, let $d_h=d_\varepsilon$, and let $\varrho=\exp(5\varepsilon\delta)/\varepsilon'$.
Before we recursively construct the subtree $T'$ of~$G$ that contains all rays of the final tree that start at the root, let us give a brief outline of that construction.
We construct $T'$ by choosing in each step $j$ a suitable subset $S_j$ of the hyperbolic boundary points with $S_{j-1}\subseteq S_j$.
The new rays are constructed to the hyperbolic boundary points in $S_j\setminus S_{j-1}$ so that for the corresponding trees $T_{j-1}$ and $T_j$ we have $T_{j-1}\subseteq T_j$.
The new rays are not constructed at the same time but one after another in a special order.
This order depends on an order of the boundary points in $S_j\setminus S_{j-1}$ that guarantees the existence of some fixed $d_j$ such that in the metric $d_h$ the distance between the new hyperbolic boundary point and the hyperbolic boundary point of the previous step we have connected this new hyperbolic boundary point to is at most~$d_j$~-- we shall substantiate the term \emph{connected} later.
In each step of the recursion there is an $\varepsilon_j$ with $\varepsilon_{j-1}>\varepsilon_j$ and with $\frac{\varepsilon_{j-2}}{\varepsilon_{j-1}}=\frac{\varepsilon_{j-1}}{\varepsilon_j}$ such that $\rand G$ is covered by the open balls of radius $\varepsilon_{j}$ and with precisely those boundary points as centers that lie in $S_{j}$.
The tree $T'$ will be the union of all~$T_j$.

\medskip

So let us start with the construction.
Since $\rand G$ has finite Assouad dimension, we may assume by Theorem~\ref{AssouadDoubling} that $\rand G$ is doubling.
As mentioned the property of the hyperbolic boundary of having finite Assouad dimension~-- and hence of being doubling~-- does not depend on the chosen metric $d_\varepsilon$ even though their actual values do.
Let $r\in VG$, and let $N=2^{\dim_2(\rand G)}$.
For the first step of the construction choose a boundary point $\mu^0\in\rand G$.
Let $S_0=\{\mu^0\}=Y_0$, let $\varepsilon_0=\diam(\rand G)$, and let $T_0$ be the graph consisting of a geodesic ray from $r$ to~$\mu^0$. Such a ray exists by Proposition~\ref{RightChoiceOfGeodRays}.

For the step $j$ of the construction let $T_{j-1}$ be the tree constructed in the previous step, let $S_{j-1}$ be the set of boundary points for which $T_{j-1}$ contains a ray converging to that boundary point, and let $\BF_{j-1}$ be the set of all closed $\varepsilon_{j-1}$-balls with centers in $Y_{j-1}$, where $Y_{j-1}\subseteq S_{j-1}$ such that we have $d_h(x,y)\geq\varepsilon_{j-1}$ for all $x,y\in S_{j-1}$ and $d_h(x,y)\geq 8N\varepsilon_{j-1}$ for all $x,y\in Y_{j-1}$.
Furthermore, we assume that $\BF_{j-1}$ is a closed cover of~$\rand G$ and that the tree $T_{j-1}$ has the following two properties.

\begin{itemize}
\item[$(*)$] Every edge in $T_{j-1}$ lies either in the tree $T_0$ or on a geodesic double ray in~$G$ between two boundary points in $S_{j-1}$.
\item[$(**)$] Every ray in $T_{j-1}$ is eventually geodesic.
\end{itemize}

Before we start with the construction of the new rays, we have to determine the new sets $Y_j$ and $S_j$ of hyperbolic boundary points and a closed cover $\BF_j$ of the hyperbolic boundary.

By Lemma \ref{LS2.3} there is a finite closed covering $\BF_{j}$ of $\rand G$ with balls of radius $\frac{\varepsilon_{j-1}}{32}$, with $\frac{\varepsilon_{j-1}}{32}$-multiplicity at most $N^4$ such that the set $Y_j$ of centers of these balls is a superset of~$S_{j-1}$, such that each two elements of $Y_j$ have distance more than $\frac{\varepsilon_{j-1}}{32}$ and such that every element of~$Y_j$ has $(\frac{3\varepsilon_{j-1}}{32})$-multiplicity at most~$N^4$ in~$\BF_j$.
Let
$$\varepsilon_j=\frac{\varepsilon_{j-1}}{64\varrho N^4}$$
and let $S_j$ be a finite subset of $\rand G$ with $S_{j-1}\cup Y_j\subseteq S_j$, with $d_h(\eta,\mu)>\varepsilon_j$ for all $\eta,\mu\in S_j$, and such that $\{\overline{B}_{\varepsilon_j}(s)\mid s\in S_j\}$ is a closed cover of~$\rand G$ with $\frac{\varepsilon_{j-1}}{8}$-multiplicity at most $N^{\log_2(32\varrho N^4)}$.
We obtain this set by applying the proof of Lemma~\ref{LS2.3}.
Indeed, we may consider in that proof the value $r$ to be $\varepsilon_j$ and $Y=Y_j$.
If we consider the balls $\overline{B}_{8\varrho N^4\varepsilon_j}(z)$, then we obtain $|S_j\cap \overline{B}_{8 \varrho N^4\varepsilon_j}(z)|\leq N^{\log_2(32\varrho N^4)}$.
Hence, $\{\overline{B}_{\varepsilon_j}(s)\mid s\in S_j\}$ has $\frac{\varepsilon_{j-1}}{8}$-multiplicity at most $N^{\log_2(32\varrho N^4)}$.

Let $T^0_j=T_{j-1}$.
Let $S_j\setminus S_{j-1}=\{\mu^j_1,\ldots,\mu^j_{\abs{S_j\setminus S_{j-1}}}\}$ with the property that all $\mu^j_i$ with $\varrho \varepsilon_{{j-1}}$-multiplicity $1$ in $\BF_{{j-1}}$ have a smaller index than those that have $(2\varrho \varepsilon_{{j-1}})$-multiplicity at most $2$ in $\BF_{{j-1}}$ but not $\varrho\varepsilon_{j-1}$-multiplicity at most $1$ in~$\BF_{j-1}$ and so on until we have those that have $(N^4\varrho\varepsilon_{{j-1}})$-multiplicity at most $N^4$ in~$\BF_{{j-1}}$ but not $((N^4-1)\varrho\varepsilon_{j-1})$-multiplicity at most $1$ in~$\BF_{j-1}$.
Since $\BF_{j-1}$ is a covering of balls of radius at most $\varrho N^4\varepsilon_{j-1}=\frac{\varepsilon_{j-2}}{32}$ of $\rand G$ of $(\frac{\varepsilon_{j-2}}{32})$-multiplicity at most $N^4$, any point in~$\rand G$ has $\varrho N^4\varepsilon_{j-1}$-multiplicity at most $N^4$ in~$\BF_{j-1}$.
Thus, we have enumerated all of $S_j\setminus S_{j-1}$.

\medskip

Having built the setting for the recursion step $j$ we construct the new rays in this step.
Let us construct the rays to the $\mu^j_i$ one by one in the order they are enumerated.
For every $\mu^j_i$ there is an $\eta\in S_{{j-1}}$ with $d_h(\mu^j_i,\eta)\leq \varepsilon_{{j-1}}$, since $\{\overline{B}_{\varepsilon_{j-1}}(s)\mid s\in S_{j-1}\}$ is a covering of~$\rand G$.
Let $\pi$ be a geodesic double ray from $\mu^j_i$ to $\eta$ such that the new ray uses a subray of the existing ray in $T^{i-1}_j$ to $\eta$.
This is possible due to Proposition \ref{RightChoiceOfGeodRays} and since the rays in $T_j^{i-1}$ are eventually geodesic either by construction in this step or by property $(**)$.
It might happen that $\pi$ intersects with $T^{i-1}_j$ non-trivially apart from the common subray to~$\eta$.
Then we just add a maximal subray of~$\pi$ to~$T^{i-1}_j$ that intersects with $T^{i-1}_j$ only in its endvertex~$x$.
This vertex $x$ has to lie either in~$T_0$ or on a geodesic double ray in~$G$ between two elements of~$S_{j-1}\cup\{\mu^j_1,\ldots,\mu^j_{i-1}\}$ by~$(*)$ and by the construction of $T^{i-1}_j$.
For at least one of the involved hyperbolic boundary points of that (double) ray in~$G$, say $\eta'$, we have $(\mu^j_i,\eta)-5\delta\leq(\mu^j_i,\eta')$ because of Proposition~\ref{geodIn4Delta} and due to the definition of hyperbolic.
So we obtain from Theorem~\ref{BoundComp} the following inequality:
$$\begin{array}{lll}
d_h(\mu^j_i,\eta')&\le&\exp(-\varepsilon(\mu^j_i,\eta'))\\
&\leq&\exp(-\varepsilon(\mu^j_i,\eta)+5\varepsilon\delta)\\
&=&\varrho\varepsilon'\exp(-\varepsilon(\mu^j_i,\eta))\\
&\leq&\varrho d_h(\mu^j_i,\eta)\\
&\leq&\varrho\varepsilon_{j-1}
\end{array}$$
We shall use the fact $d_h(\mu^j_i,\eta')\leq \varrho\varepsilon_{j-1}$ later in the proof.
By adding the above mentioned subray of~$\pi$ to $T^{i-1}_j$ we obtain the tree $T^i_j$.

Let $T_j$ be the union of all $T^i_j$, in other words
$$T_j=T^{\abs{S_j\setminus S_{j-1}}}_j.$$
By construction it is clear that $(*)$ and $(**)$ hold for $T_j$ and that $T_j$ is a tree.
Let $T'=\bigcup_{i\in\nat}T_i$. Since all $T_i$ are trees and $T_{i-1}\subseteq T_i$ for all $i\in\nat$, we conclude that $T'$ is connected and cannot contain a (finite) cycle as this would already lie in some~$T_j^i$, so $T'$ is a tree.

\medskip

Let us return to the settings of the recursion step $j$ to define some terminology, which we shall use in the remainder of the proof.
So we have the hyperbolic boundary points $\mu^j_i,\eta,\eta'$ and the geodesic double ray $\pi$ from $\mu^j_i$ to~$\eta$.
Let us call $\mu^j_i$ {\em connected to} $\eta$ if $\pi$ intersects with $T^{i-1}_{j}$ only on the common subray to $\eta$ and {\em connected to} $\eta'$ else.
If $\mu^j_i$ is connected to $\eta$ then $\mu^j_i$ is {\em eventually connected to} $\eta$.
If $\mu^j_i$ is connected to $\eta'$ and $\eta'\in S_{j-1}$ then $\mu^j_i$ is {\em eventually connected to} $\eta'$, and finally, if $\mu^j_i$ is connected to $\eta'$ but $\eta'\not\in S_{j-1}$ then $\mu^j_i$ is {\em eventually connected to} the same boundary point $\eta'$ is eventually connected to.

\medskip

We have just constructed a subtree $T'$ of~$G$ that does not necessarily contain every vertex of~$G$.
So, in general, it is not a spanning tree of~$G$.
In the next part of the proof, we connect all vertices of $G-T'$ to~$T'$, recursively, to get a new tree $T$, which will be a spanning tree of~$G$.
Let $T'_0:=T'$.
We connect the new vertices without creating new rays as follows.
First we can easily extend the tree by adding all finite components of $G-T'$ to $T'$.
Then we add every vertex with distance $d(r,G-T'_0)$ to $T'$ by a path lying outside $B_{d(r,G-T'_0)}(r)$.
There might be vertices for which there exists no such path.
Then we do not add these.
Let $T_1'$ be the new tree.
If there is a vertex in $G-T_1'$ with distance $d(r,G-T_0')$ that does not lie in any finite component of $G-T_1'$, then there is a path from such a vertex to $T'_1-B_{d(r,G-T'_0)}(r)$ that intersects with $T'_1$ trivially except for its endvertex, because $G$ is locally finite.
This path has a last vertex $x$ with distance $d(r,G-T'_0)$ to~$r$.
But this is a contradiction, since $x$ had to be added with a path to~$T'$.
So all the vertices in $G-T_1'$ with distance $d(r,G-T_0')$ to~$r$ lie in finite components of $G-T_1'$.
For the following step we keep in mind the largest distance $d_1$ from $r$ to a vertex lying in $T_1'-T'_0$.
In the recursion step $i$ we add all finite components of $G-T_i'$.
Then we add for each vertex $x$ with $d(r,x)=d_{i-1}+1$ a path to~$T_i'$ that lies completely outside~$B_{d_1}(r)$.
Let $T'_{i+1}$ be the new tree.
Once again, there might be vertices $x$ with $d(r,x)=d_{i-1}+1$ that cannot be connected to $T_i'$ in such a way.
These will be treated at the beginning of the next step of the recursion, as they lie again in finite components of $G-T_{i+1}'$.

Let $T=\bigcup_{i\in\nat}T_i'$.
Obviously, $T$ is a spanning tree of~$G$.
Furthermore, there is no ray in $T-T'$ as we have not connected any vertex in step $j$ to a vertex outside $T'$ except for those in finite components of $G-T_i'$.
Thus, to prove the properties (i) to (iii) of Theorem~\ref{SpanningTrees} for~$T$, it suffices to prove them for $T'$.

\bigskip

We continue with the next part of the proof by showing two claims about the sets $S_j$ that we used during the construction of~$T'$.

\begin{Claim}\label{1stClaim}
Let $\mu^j_{i_1}$ and $\mu^j_{i_2}$ be elements of $S_j\setminus S_{j-1}$ with $d_h(\mu^j_{i_1},\mu^j_{i_2})\leq \varrho\varepsilon_{j-1}$ such that both do not have $((n-1)\varrho\varepsilon_{j-1})$-multiplicity at most $n-1$ but $(n\varrho\varepsilon_{j-1})$-multiplicity at most $n$ in $\BF_{j-1}$ for some $n\leq N^4$.
Then for any $B\in\BF_{j-1}$ with $d_h(\mu^j_{i_1},B)\leq n\varrho\varepsilon_{j-1}$, we have $d_h(\mu^j_{i_2},B)\leq n\varrho\varepsilon_{j-1}$ and vice versa.
\end{Claim}

\begin{proof}[Proof of Claim~\ref{1stClaim}]
Since the $((n-1)\varrho\varepsilon_{j-1})$-multiplicity of both $\mu^j_{i_1}$ and $\mu^j_{i_2}$ in~$\BF_{j-1}$ must be $n$, every element of~$\BF_{j-1}$ with distance at most $n\varrho\varepsilon_{j-1}$ to $\mu^j_{i_k}$ has distance at most $((n-1)\varrho\varepsilon_{j-1})$ to~$\mu^j_{i_k}$ and thus distance at most $n\varrho\varepsilon_{j-1}$ to $\mu^j_{i_\ell}$ with $k\neq \ell$.
So it counts for the $(n\varrho\varepsilon_{j-1})$-multiplicity of~$\mu_{i_\ell}^j$ in~$\BF_{j-1}$.
\end{proof}

\begin{Claim}\label{2ndClaim}
\begin{enumerate}[(i)]
\item If $\mu^j_{i}$ is connected to $\mu\in S_j$ in~$T^i_j$, then we have
\[d_h(\mu^j_{i},\mu)\leq \varrho\varepsilon_{j-1}.\]
\item If $\mu^j_{i}$ is eventually connected to $\eta\in S_{j-1}$ in~$T_j$, then we have
\[d_h(\mu^j_{i},\eta)\leq \varrho N^4\varepsilon_{j-1}+\diam(\BF_{j-1})\leq 5\varrho N^4\varepsilon_{j-1}.\]
Furthermore, $\eta$ lies in some $B\in\BF_{j-1}$ with $d_h(\mu^j_{i},B)\leq \varrho N^4\varepsilon_{j-1}$.
\end{enumerate}
\end{Claim}

\begin{proof}[Proof of Claim~\ref{2ndClaim}]
The first statement holds immediately, as we mentioned during the construction.
So let $\mu^j_{i}$ be eventually connected to~$\eta$.
If $\mu^j_{i}$ has $\varrho\varepsilon_{j-1}$-multiplicity $1$ in $\BF_{j-1}$, then it can only be connected to a boundary point $\mu$ with $d_h(\mu^j_{i},\mu)\leq\varrho\varepsilon_{j-1}$ by the construction.
Both these boundary points must lie in the same $B\in\BF_{j-1}$ by the $\varrho\varepsilon_{j-1}$-multiplicity of~$\mu_{i}^j$ in~$\BF_{j-1}$ and as $\BF_{j-1}$ covers~$\rand G$.
Thus and by induction, $\mu^j_{i}$ is eventually connected to some boundary point $\eta$ in the same $B\in\BF_{j-1}$ that contains $\mu^j_{i}$ such that
$$d_h(\mu^j_{i},\eta) \leq \diam(\BF_{j-1})\leq 4\varrho N^4\varepsilon_{j-1}.$$
Let us now assume that $\mu^j_{i}$ has $(k\varrho\varepsilon_{j-1})$-multiplicity at most $k$ in $\BF_{j-1}$ but not $((k-1)\varrho\varepsilon_{j-1})$-multiplicity at most $k-1$ in $\BF_{j-1}$.
Originally, we wanted to connect $\mu^j_{i}$ to a boundary point $\mu\in S_{j-1}$ with $d(\mu,\eta)\leq\varepsilon_{j-1}$.
But by our construction $\mu^j_{i}$ is connected to a boundary point $\nu\in S_{j-1}\cup\{\mu^j_1,\ldots,\mu^j_{i-1}\}$ with $d_h(\mu^j_{i},\nu)\leq\varrho\varepsilon_{j-1}$ and thus, $\nu$ is contained in an element $B\in\BF_{j-1}$ which is responsible for the $(k\varrho\varepsilon_{j-1})$-multiplicity of $\mu^j_{i}$ in~$\BF_{j-1}$.
If $\nu\in S_{j-1}$, then nothing remains to show.
So we may assume that $\nu\in S_j\setminus S_{j-1}$.
Let $\eta$ be the boundary point $\nu$ is eventually connected to.
By induction we know that $\eta$ lies in one of the elements of $\BF_{j-1}$, say in $B$, that is responsible for the $(n\varrho\varepsilon_{j-1})$-multiplicity of at most $n$ of~$\nu$ for some $n\leq N^4$.
As $\nu$ has a smaller index than $\mu^j_{i}$, we have $n\leq k$.
Similar to Claim~\ref{1stClaim}, we know that $B$~is responsible for the $(k\varrho\varepsilon_{j-1})$-multiplicity of at most $k$ of~$\mu^j_{i}$.
We conclude
$$d_h(\mu^j_{i},\eta)\leq k\varrho\varepsilon_{j-1}+\diam(\BF_{j-1})\leq k\varrho\varepsilon_{j-1}+4\varrho N^4\varepsilon_{j-1}.$$
As all elements of $S_j\setminus S_{j-1}$ have $\varrho N^4\varepsilon_{j-1}$-multiplicity at most $N^4$ in~$\BF_{j-1}$, Claim~\ref{2ndClaim} follows.
\end{proof}

The last part of the proof is to show the properties (i) to (iii) for~$T'$.
For a closed ball $B_k\in\BF_k$, let $B'_k$ denote $B_k$ together with all other (at most $N^4$) closed balls in $\BF_k$ with distance at most $\varrho N^4\varepsilon_k$ to~$B_k$.

Since $(**)$ holds in each recursion step of the construction of~$T'$, all we have to prove for (i) is that every ray we created by the construction of infinitely many rays converges to some boundary point.
Let us assume that $\pi$ is a ray in~$T'$ with the property that there exists infinitely many finite subpaths $(P_i)_{i\in\nat}$ of~$\pi$ such that each $P_i$ was used for the construction of another ray~$R_i$.
Let $\eta_i$ be the limit point of~$R_i$.
Since $\widehat{G}$ is compact, $\pi$ has at least one accumulation point $\eta$ in $\rand G$.
Thus, we have to prove that there exists no second accumulation point.
For every step~$k$, let $B_k$ be one of the closed balls in~$\BF_k$ that contain~$\eta_i$ if $\eta_i$ is the last element in the enumeration of $S_k\setminus S_{k-1}$ and let $B_k=B_{k-1}$ if no $\eta_i$ lies in $S_k\setminus S_{k-1}$ (with $B_0=\rand G$).
Any second boundary point must lie~-- like $\eta$ does~-- in $\bigcap_{k\in\nat}B'_k$ by Claim~\ref{2ndClaim}.
Since $\bigcap_{k\in\nat}B_k'$ is a set with at most one element, $\pi$ has precisely one accumulation point.

\smallskip

For the proof of~(ii), let $\eta$ be a boundary point of~$G$.
Then for each $k$, there is at least one closed ball $B_k$ in the step $k$ of the construction of~$T'$ with $\eta\in B_k$.
Thus, there is a boundary point $\eta_k\in S_k\cap B_k$ with $d_h(\eta_k,\eta)\leq\varepsilon_k$ and we constructed a ray to~$\eta_k$.
Since $G$ is locally finite, there is an infinite path $\pi$ such that each edge of that path is contained in infinitely many of the rays to the boundary point~$\eta_k$.
By the Gromov-product, Claim~\ref{2ndClaim}, and the choice of the rays to the~$\eta_k$, the ray~$\pi$ must have $\eta$ as an accumulation point.
Due to~(i), $\pi$ converges to~$\eta$.

\smallskip

To any closed ball $B\in\BF_k$ in step $k$ there are at most $N^4$ closed balls in the step $k-1$ to which boundary points of~$B$ are eventually connected to due to the choice of~$\BF_k$ and Claim~\ref{2ndClaim}.
Additionally, each ball contains at most $N^{\log_2(32\varrho N^4)}$ many elements of~$S_k$.
Thus, the number of rays to one boundary point is bounded by a function depending only on the doubling property of~$\rand G$.
Since for given $\varepsilon$, the doubling property depends only on the Assouad dimension, this proves the only remaining part (iii) of Theorem \ref{SpanningTrees}.
\end{proof}

A graph $G$ has {\em bounded growth at some scale} if there are constants $r,R$ with $R>r>0$ and $N\in\nat$ such that every ball of radius less than $R$ can be covered by $N$ balls of radius less than $r$.

Bonk and Schramm proved the following theorem about hyperbolic graphs with bounded growth at some scale.

\begin{Tm}\label{BS9.2}{\rm \cite[Theorem 9.2]{BonkSchramm_Embeddings}}
Let $G$ be a hyperbolic graph with bounded growth at some scale.
Then the hyperbolic boundary of~$G$ is doubling and has finite Assouad dimension.\qed
\end{Tm}

Therefore, we obtain the following corollary.

\begin{Cor}
Let $G$ be a locally finite hyperbolic graph with bounded growth at some scale. Then there exists an $n\in\nat$ and a rooted spanning tree $T$ of~$G$ with the following properties:
\begin{enumerate}[(i)]
\item Every ray in $T$ converges to a point in the boundary of~$G$;
\item for every boundary point $\eta$ of~$G$ there is a ray in~$T$ converging to~$\eta$;
\item for every boundary point $\eta$ of~$G$ there are at most $n$ distinct rays in~$T$ starting at the root of~$T$ and converging to~$\eta$.\qed
\end{enumerate}
\end{Cor}

Since all graphs with bounded degree have bounded growth at some scale, all almost transitive graphs and in particular all Cayley graphs fulfill the assumptions of Theorem~\ref{SpanningTrees}.

\smallskip

As mentioned in the introduction, arbitrary locally finite hyperbolic graphs do not have spanning trees which are faithful to hyperbolic boundary points instead of ends.
Let us discuss an explicit example:

\begin{Exam}\label{ex_introEx}
Let $G$ be the graph of Figure~\ref{MG}.
\begin{figure}[h]
\begin{center}
\includegraphics{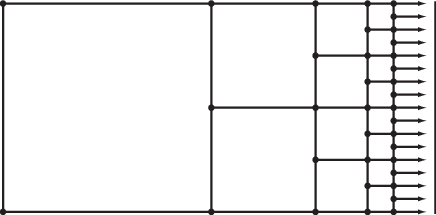}
\caption[Figure 1]{A hyperbolic graph with its hyperbolic boundary}\label{MG}
\end{center}\end{figure}

Its hyperbolic boundary is homeomorphic to the real unit interval.
Now suppose there is a spanning tree $T$ of~$G$ with precisely one ray from the root to each boundary point of~$G$.
Then there is a vertex $x$ that separates $T$ into at least two infinite components, call them $C_1,\ldots,C_n$.
For each~$i$ let $Z_i\subseteq \rand G$ be the set of boundary points to which there is a ray in~$C_i$.
As $G$ is locally finite, it is not hard to see that the sets~$Z_i$ are closed in~$\rand G$.
So they have to intersect, since $\rand G$ is connected and $\bigcup Z_i=\rand G$.
\end{Exam}

\providecommand{\bysame}{\leavevmode\hbox to3em{\hrulefill}\thinspace}
\providecommand{\MR}{\relax\ifhmode\unskip\space\fi MR }
\providecommand{\MRhref}[2]{%
  \href{http://www.ams.org/mathscinet-getitem?mr=#1}{#2}
}
\providecommand{\href}[2]{#2}

\end{document}